%% file: coarse_inv_spec.tex
\documentclass[a4paper]{article}
\usepackage[final]{microtype}
\usepackage{lipsum}
\usepackage{expl3, etoolbox}

\input{GenCommands.tex}

\input{CrsCommands.tex}

\ExplSyntaxOn
\clist_map_inline:nn
{example, theorem, lemma, remark, question, corollary, proof}
{
    \AfterEndEnvironment{#1}{\pagebreak[1]}  
}
\ExplSyntaxOff

\title{On a coarse invertibility spectrum for coarse groups}
\author{Leo Schäfer and Federico Vigolo}

\newcommand*{\minvariant}[2][\mathcal P]{{\mathfrak S_{#1}\paren{#2}}}

\begin{document}
\maketitle

\begin{abstract}
  We introduce a coarse algebraic invariant for coarse groups and use it to differentiate various coarsifications of the group of integers. This lets us answer two questions posed by Leitner and the second author. The invariant is obtained by considering the set of exponents $n$ such that taking $n$-th powers defines a coarse equivalence of the coarse group.
\end{abstract}

\section{Introduction}
\subsection{A number-theoretic problem}\label{ssec:intro: nr theory}
One interesting problem is to understand the large scale geometry of the metric spaces obtained by equipping the set of integers with the word metric associated with some geometric series. Namely, given $g \in \mathbb N$ with $g \geq 1$ let 
\(
  S_g = \{\pm g^n : n \in \mathbb N\}.
\)
The \emph{word length} $\ell_g(n)$ of an integer $n\in\mathbb Z$ is the minimal number of elements in $S_g$ adding up to $n$ (with multiplicity), and the \emph{word metric} $d_{g}(n,m)\coloneqq \ell_g(n-m)$ is the induced addition-invariant distance.

For $g=1$, $d_g$ is just the usual Euclidean metric, while for $g\geq 2$ the distance $d_{g}(n,m)$ can be seen as a ``measure of difference'' between the $g$-adic series representations of $n$ and $m$.

Observe that any $g \geq 2$ gives rise to a non-proper metric. This shows that there is a very clear qualitative difference between $d_1$ and all the other metrics $d_g$. Understanding the relation between $d_{g_1}$ and $d_{g_2}$ for two different values $g_1,g_2\geq 2$ is a more delicate matter: understanding whether they are bi-Lipschitz equivalent already requires some work, and goes back to a theorem by M.~Nathanson.

\begin{theorem}[\cite{nathanson2011bi}]\label{thm: nathanson bilip equivalence}
  Given $g_1,g_2\geq 2$ the identity map $(\mathbb Z,d_{g_1})\to (\mathbb Z,d_{g_2})$ is a bi-Lipschitz equivalence if and only if $g_1^n=g_2^m$ for some $n,m\in\mathbb N$.
\end{theorem}

Notice however, \cref{thm: nathanson bilip equivalence} does not rule out the possibility that there exist some bi-Lipschitz function $f\colon(\mathbb Z, d_{g_1})\to(\mathbb Z,d_{g_2})$ that is not the identity map. Namely, it does not answer the following.

\begin{question}\label{qu:intro:metrics}
  Is it true that $(\mathbb Z, d_{g_1})$ and $(\mathbb Z,d_{g_2})$ are bi-Lipschitz equivalent metric spaces if and only if $g_1^n= g_2^m$ for some $n,m\in\mathbb N$?
\end{question}

\cref{qu:intro:metrics} is a difficult question because only few tools are available to study non-proper metric spaces of this sort.
As a matter of fact, it is already a known open problem to decide whether $(\mathbb Z, d_{2})$ and $(\mathbb Z,d_{3})$ are bi-Lipschitz equivalent (attributed to R.\ Schwartz \cite{nathanson2011problems}*{Problem 6}). 

\smallskip

In this work we investigate a question related to \cref{qu:intro:metrics}, but where we do not entirely forget the group structure of $\mathbb Z$. Since $\pm\id$ are the only automorphisms of $\mathbb Z$, it is clear that one cannot produce a bi-Lipschitz group isomorphism $f\colon(\mathbb Z, d_{g_1})\to(\mathbb Z,d_{g_2})$ unless $g_1^n=g_2^m$. But what happens if one only requires the addition to be preserved up to some bounded error?

Namely, we are interested in bi-Lipschitz mappings $f\colon(\mathbb Z, d_{g_1})\to(\mathbb Z,d_{g_2})$ such that 
\begin{equation}\label{eq: quasimorphism}
  d_{g_2}\bigl(f(a+b), f(a)+f(b)\bigr)\leq C  
\end{equation}
for some arbitrarily large but fixed constant $C\geq 0$. Such a function $f$ is called \emph{coarse isomorphism}.
Notice that the extra freedom given by the additive constant $C$ allows for a great flexibility. In the very simple case of $d_1$, the set of mappings $f\colon (\mathbb Z, d_{1})\to(\mathbb Z,d_{1})$ satisfying \eqref{eq: quasimorphism} are known as quasi-morphisms. It is immediate to verify that any fixed $\lambda\in \mathbb R$ gives rise to a quasimorphism via $n\mapsto\floor{\lambda n}$ (conversely, one can also show that every quasi-morphism is close to one such multiplication map).
Classifying all coarse automorphisms of $(\mathbb Z, d_{g})$ for $g>1$ seems to be a very challenging problem.

The main application of the present paper makes substantial progress in the direction of answering \cref{qu:intro:metrics} for coarse isomorphism. Namely, we prove the following.

\begin{alphthm}\label{thm:intro}
  Let $g_1,g_2 \in \mathbb N$ and assume the prime divisors of $g_1$ and $g_2$ do not coincide.
  Then $(\mathbb Z,d_{g_1})$ and $(\mathbb Z,d_{g_2})$ are not coarsely isomorphic.
\end{alphthm}

In particular, it follows that $(\mathbb Z,d_{2})$ and $(\mathbb Z,d_{3})$ are not coarsely isomorphic, which answers \cite{coarse_groups}*{Question 9.3.1}.
To illustrate the technique of proof and further motivate the interest in coarse isomorphisms, it is necessary to say a few words about the theory of coarse groups.

\subsection{Context: coarse groups}
Without entering the details, a \emph{coarse space} $\crse X$ is a set $X$ together with a \emph{coarse structure} $\mathcal E$. That is, $X$ is a set where there is a well defined notion of ``uniform boundedness''. Typical examples are given by metric spaces---which we also call \emph{metric coarse spaces}: in this case uniform boundedness simply means ``uniformly bounded diameter''.

A \emph{coarse map} $\crse{f\colon X\to Y}$ is an equivalence class of controlled functions. Here, a function $f\colon X\to Y$ is said to be \emph{controlled} if it sends uniformly bounded subsets to uniformly bounded subsets, and two functions $f,f'$ are equivalent if they are \emph{close} (that is, the images $f(x)$ and $f'(x)$ stay uniformly close to one another as $x\in X$ varies).
Coarse spaces and coarse maps are the objects and morphisms of the \emph{category of coarse spaces}. A \emph{coarse equivalence} is a coarse map that has a coarse inverse (\emph{i.e.}\ an isomorphism in the category of coarse spaces). A typical example of coarse equivalence is given by (equivalence classes of) quasi-isometries between metric spaces.

A \emph{coarse group} is a coarse space $\crse G$ together with a coarse map $\crse{\cop\colon G\times G\to G}$ that satisfies the group axioms within the \emph{coarse category} (so that, for instance, associativity only needs to hold up to uniformly bounded error). In other words, a coarse group is a group object in the category of coarse spaces. In this work we will be mostly concerned with some rather concrete examples of coarse groups, and our arguments require very little coarse geometric background. For this reason, we will avoid formally introducing all the relevant notions, and refer the interested reader to \cite{coarse_groups} for an in-depth introduction to the subject.

A \emph{coarse homomorphism} between two coarse groups $\crse G$, $\crse H$ is a coarse map $\crse{f\colon G\to H}$ that commutes with the coarse-group operations up to uniformly bounded error (\emph{i.e.}\ such that $f(g)\ast f(g')$ stays uniformly close to $f(g\ast g')$). It is a \emph{coarse isomorphism} if it is coarsely invertible. In other words, coarse homomorphisms and isomorphisms are just homomorphism and isomorphism among the group objects of the coarse category.

Returning to $(\mathbb Z,d_g)$: the addition $+\colon (\mathbb Z,d_g)\times (\mathbb Z,d_g)\to (\mathbb Z,d_g)$ is a Lipschitz map, and therefore controlled. This means that $(\mathbb Z,d_g)$ can be thought of as a metric coarse group, where $+$ is a representative for the coarse operation $\crse *$ (this is always the case for groups equipped with bi-invariant metrics).
From this perspective, condition \eqref{eq: quasimorphism} means precisely that $f(g)\ast f(g')$ stays uniformly close to $f(g\ast g')$. If $f$ is also Lipschitz, this is equivalent to saying that $\crse f$ is a coarse homomorphism. Conversely, it is not hard to show that for any $g_1,g_2\geq 2$ a coarse isomorphism $\crse f\colon (\mathbb Z,d_g)\to (\mathbb Z,d_g)$ in the sense of coarse groups must admit a bi-Lipschitz representative $f$ (see \cref{appendix}), which is hence a coarse isomorphism in the sense of \cref{ssec:intro: nr theory}.

Rephrasing it in this context, the original number-theoretic \cref{qu:intro:metrics} is precisely asking what values of $g$ give rise to coarsely equivalent spaces, and the variant that we investigate in this work asks the same question about coarsely isomorphic groups.

\subsection{Technique of proof}
To prove \cref{thm:intro}, we introduce a new coarse-algebraic invariant that is simple enough to compute and refined enough to actually tell apart some coarse groups.
The idea is to consider the set of exponents $n$ such that taking the $n$-th power defines a coarse equivalence: this ``spectrum'' is easily shown to be invariant under coarse isomorphisms (\cref{sec:spectra}).
We are then able to completely compute this invariant for $(\mathbb Z,d_g)$ (\cref{thm:invariant for Z Sq}), and \cref{thm:intro} is a direct consequence of this computation.

We are also able to compute this invariant for a different class of coarsifications of $\mathbb Z$ defined using profinite topologies (\cref{thm:invertible iff p in Q}. See \cref{sec:profinite} for details). It turns out that in this case the invariant is sufficiently refined to tell any two profinite coarsifications apart, and results in the following theorem.

\begin{alphthm}\label{thm:intro2}
  If $Q\neq Q'$ are two different sets of primes and $\varcrs{Q}$, $\varcrs{Q'}$ are the associated profinite coarse structures on $\mathbb Z$, then the coarse groups $(\mathbb Z,\varcrs{Q})$ and $(\mathbb Z,\varcrs{Q'})$ are not coarsely isomorphic.
\end{alphthm}

\cref{thm:intro,thm:intro2} respectively answer Questions 9.3.1 and 9.3.2 in \cite{coarse_groups}.

\begin{remark}
  It is worthwhile to remark that these results are perhaps the first meaningful example where the coarse algebraic structure has really been used to tell some coarse groups apart: up to this point, the main avenue for distinguishing coarse groups was to use purely geometric means.
  Such a geometric approach would have been hard to use in the current setting, as these are very complicated metric spaces (recall that it is still very much unclear whether $(\mathbb Z,d_2)$ and $(\mathbb Z,d_3)$ can be distinguished from one another by purely geometrical means).
\end{remark}

\medskip

\noindent{\bf Structure of the paper.} In \cref{sec:spectra} we define the spectra of power invertibility for general coarse groups. In \cref{sec:metric,sec:profinite} we compute them for the coarse groups of the from $(\mathbb Z,d_g)$ and $(\mathbb Z,\varcrs{Q})$ respectively, thus proving \cref{thm:intro,thm:intro2}. We conclude by listing some open problems in \cref{sec:problems}.

\medskip

\noindent{\bf Acknowledgements.}
We are grateful to the anonymous referee for pointing out an inaccuracy in the original version of \cref{prop:nathanson for Zq}.
This work is funded by the  Deutsche Forschungsgemeinschaft (DFG) as part of the GRK 2491:  Fourier Analysis and Spectral Theory -- Project-ID 398436923.

\section{Spectra of power invertibility}\label{sec:spectra}
Let $\crse G = (G,\mathcal E)$ be a coarse group. We consider the $n$-th power map
\[
  \begin{tikzcd}[row sep = 0]
    \crse\mu_n\colon        \crse G \arrow[r] & \crse G \\
    \phantom{\crse\mu_n\colon} g \arrow[|->,r] & \underbrace{g\ast\cdots \ast g}_{n \text{ times}}
  \end{tikzcd}
\]
($\crse\mu_0$ is the map sending the whole of $\crse G$ to the neutral element).
Observe that $g\ast\cdots \ast g$ is not a well-defined element of $G$, because the operation $\ast$ is only coarsely associative. However, the points obtained by choosing different orders of associations are uniformly close to one another and therefore $\crse\mu_n$ is coarsely well-defined.
More abstractly, $\crse\mu_n$ can be defined as the composition of the diagonal embedding with the appropriate number of iterations of the group operation 
\begin{equation}\label{eq:mu as composition}
  \crse G\xrightarrow{\quad\Delta\quad} \crse G^n \xrightarrow{\ \cop\cdots\cop\ }\crse G.
\end{equation}
This approach makes it clear that $\crse \mu_n$ is a coarse map. We are interested in the following set:
\begin{equation*}
  \minvariant[\mathbb N]{\crse G}\coloneqq\braces{n\in\mathbb N : \crse \mu_n\text{ is a coarse equivalence}}.
\end{equation*}

Observe that for every $m,n\in\mathbb N$ we have
\begin{equation}\label{eq: mu is multiplicative}
  \crse \mu_n\circ\crse \mu_m = \crse \mu_{nm} = \crse \mu_m\circ\crse \mu_n.
\end{equation}
This has some interesting consequences. Since compositions of coarse equivalences are coarse equivalences, it follows that $\minvariant[\mathbb N]{\crse G}$ is closed under multiplication. More interestingly, it is also closed under division. Namely, if $n \in \minvariant[\mathbb N]{\crse{G}}$ and $m \mid n$, then the composition $\crse\mu_{n/m}\circ \crse\mu_{n}^{-1}$ is a coarse inverse for $\crse\mu_m$, hence $m\in \minvariant[\mathbb N]{\crse{G}}$.
In other words, if we let 
\begin{equation*}
  \minvariant{\crse G}\coloneqq\braces{p\text{ prime} : \crse\mu_p\text{ is a coarse equivalence}}
\end{equation*}
we see that $\minvariant[\mathbb N]{\crse G}$ is the sub-semigroup of $\mathbb N$ generated by $\minvariant{\crse G}\cup \{1\}$.

\begin{remark}
  Technically, the statement about semigroup generation is slightly problematic if one includes $0$ in their definition of $\mathbb N$---as we do. Observe however, that $0$ is in $\minvariant[\mathbb N]{\crse G}$ if and only if $\crse G$ is bounded as a coarse space or, equivalently, it is coarsely isomorphic to the trivial coarse group $\{1\}$. This seems a good enough reason not to worry about this issue here.
\end{remark}

One may extend the definition of $\crse{\mu}_{\variable}$ to negative integers by taking inverses:
\[
\begin{tikzcd}
  \crse G \arrow{r}{\cinversefn} \arrow[bend right=25,swap]{rrr}{\crse{\mu}_{-n}}&
  \crse G\arrow{r}{\quad\Delta\quad} &
  \crse G^n \arrow{r}{\ \cop\cdots\cop\ } &
  \crse G.
\end{tikzcd}
\]
It is then natural to set
\begin{equation*}
  \minvariant[\mathbb Z]{\crse G}\coloneqq\braces{k\in\mathbb Z : \crse \mu_k\text{ is a coarse equivalence}}.
\end{equation*}
However, no additional information is gained by doing so. In fact, observe that \eqref{eq: mu is multiplicative} remains true with $m,n\in\mathbb Z$ and that $\crse \mu_{-1}$ is always a coarse equivalence because $\cinversefn$ is an involution. It then follows that $\minvariant[\mathbb Z]{\crse G} = \minvariant[\mathbb N]{\crse G}\cup -\minvariant[\mathbb N]{\crse G}$.

It is also rather natural to consider the set
\[
    \minvariant[\mathbb Q]{\crse G} = \bigl\{ \tfrac{p}{q} : p,q \in \minvariant[\mathbb Z]{\crse G},\ q \neq 0 \bigr\}
\]
of rational exponents that define a coarse equivalence, or its restriction to the multiplicative group $\mathbb Q_+$ of strictly positive rationals 
\[
  \minvariant[\mathbb Q_+]{\crse G}\coloneqq \minvariant[\mathbb Q]{\crse G}\cap \mathbb Q_+.
\]
The former is a multiplicative group whenever $\crse G$ is non-trivial (so that $0\notin\minvariant[\mathbb Q]{\crse G}$), the latter is always a group.
Alternatively, $\minvariant[\mathbb Q]{\crse G}$ consists of the set of rational numbers $p/q$ (with $p,q$ coprime) such that $\crse\mu_{p/q}\coloneqq \crse\mu_p\circ\crse\mu_q^{-1}$ is well-defined and coarsely invertible.

Once again, no new information is gained. In fact, $\minvariant[\mathbb Q_+]{\crse G}\leq\mathbb Q_+$ is the multiplicative subgroup generated by $\minvariant{\crse G}$ and, if $\crse G$ is not trivial, we have 
\[ \minvariant[\mathbb Q]{\crse G}=\minvariant[\mathbb Q_+]{\crse G}\cup -\minvariant[\mathbb Q_+]{\crse G}. \]
Vice versa, since $\minvariant[\mathbb N]{\crse G}$ is closed under division, we have that $\minvariant[\mathbb N]{\crse G}=\minvariant[\mathbb Q]{\crse G}\cap\mathbb N$.

\medskip

In summary, all the sets $\minvariant[\!\variable\!]{\crse G}$ above defined essentially contain the same amount of information about the coarse algebraic structure of $\crse G$.\footnote{%
The only discrepancies appear if $\crse G$ is the trivial coarse group: since triviality is witnessed by coarse invertibility of $\crse\mu_0$, the spectra $\minvariant[\mathbb N]{\crse G},\minvariant[\mathbb Z]{\crse G},\minvariant[\mathbb Q]{\crse G}$ recognize whether $\crse G$ is trivial, while $\minvariant{\crse G},\minvariant[\mathbb Q_+]{\crse G}$ do not.
}
We call them \emph{spectra of power invertibility}. Which spectrum is most convenient to use will presumably vary depending on the situation. The best choice for the rest of this note will be the prime one $\minvariant{\crse G}$.

\medskip

Let now $\crse G$, $\crse H$ be two arbitrary coarse groups, and let $\crse{\mu^{G}}_n$, $\crse{\mu^{H}}_n$ denote their respective power functions. Almost by definition, if $\crse{f\colon G\to H}$ is a coarse homomorphism then
\[
  \crse f \circ \crse{\mu^{G}}_n = \crse{\mu^{H}}_n \circ \crse f.
\]
Since coarse isomorphisms are coarse equivalences, it follows that the spectra of power invertibility are indeed invariant under coarse isomorphism. The following example is due:

\begin{example}
  Consider $\mathbb Z$ with the metric defined by the Euclidean absolute value $\abs\variable$ (\emph{i.e.} the metric $d_1$ from the introduction).
  We may then see $(\mathbb Z,\abs{\variable})$ as a coarse group, and in this case $\crse \mu_n$ is nothing but (the equivalence class of) the multiplication function $\mu_n\colon k\mapsto nk$.
  It is evident that for every $n\in\mathbb N$ the function $\mu_n$ is a coarse equivalence, \emph{e.g.}\ the function $k\mapsto \floor{k/n}$ is a coarse inverse for it.
  This shows that $\minvariant{\mathbb Z,\abs{\variable}}=\mathcal P$ is the set of all primes.

  An alternative way to compute this is to note that $(\mathbb Z,\abs{\variable})$ is coarsely isomorphic to $(\mathbb R,\abs{\variable})$ (consider the inclusion and the floor function). Then $\minvariant[Q_+]{\mathbb Z,\abs{\variable}}=\minvariant[Q_+]{\mathbb R,\abs{\variable}}$, and the latter is obviously the whole of $\mathbb Q_+$.
\end{example}

Before returning to the number-theoretic problems of the introduction, we end this section with a few remarks.

\begin{remark}
  A coarse group $\crse G$ is \emph{coarsely abelian} if $g_1\ast g_2$ and $g_2\ast g_1$ stay uniformly close to one another as $g_1,g_2\in \crse G$ vary. As to be expected, coarsely abelian coarse groups are much easier to investigate than their general counterpart. One useful remark is that if $\crse G$ is coarsely abelian then power maps $\crse \mu_n$ are coarse homomorphisms for every $n\in \mathbb N$.
  It is also easy to see that the converse holds. Namely, if $\crse\mu_2$ is a coarse homomorphism then $\crse G$ is coarsely abelian.
\end{remark}

\begin{remark}
  Computing power invertibility spectra for coarse groups that are not coarsely abelian is a considerably harder task, and it seems legitimate to expect that they will often be trivial.
  The world of groups is however large and wild, so it should be possible to construct examples with interesting power invertibility spectra.
  For instance, it is proved in \cite{guba1987finitely} that there exists a non-abelian group $G$ with the property that every element $g\in G$ admits a unique $n$-th root. It seems plausible that such a group admits a connected coarse structure $\mathcal E$ such that $\crse G = (G,\mathcal E)$ is a coarse group that is not coarsely abelian and has $\minvariant{\crse G}=\mathcal P$.

  In the above, a coarse structure on $G$ is \emph{connected} if every finite subset of $G$ is bounded. From a coarse algebraic point of view, these are the most interesting examples of coarse groups, because every coarse group can be decomposed into a coarsely connected coarse group and a \emph{trivially coarse} group (\emph{i.e.} a group equipped with the trivial coarse structure where the only bounded sets are singletons), see \cite{coarse_groups}*{Corollary 7.3.2}.
 \end{remark}

\begin{remark}\label{rmk: crse equivalence are proper}
  A coarse equivalence $\crse{f\colon X\to Y}$ is always \emph{proper}. Explicitly, if $f\colon X\to Y$ is a representative for $\crse f$ then it must be the case that for every bounded set $B\subseteq Y$ the preimage $f^{-1}(B)$ must also be bounded. This simple observation is all we will need in the proofs of \cref{thm:invariant for Z Sq,thm:invertible iff p in Q}.
\end{remark}

\begin{remark}
  One may define analogous invariants by considering different properties of $\crse \mu_n$. For example, $\crse \mu_n$ being a coarse embedding would be more or less analogous to the (coarse) uniqueness of $n$-th roots.
  
  Alternatively, one may also consider coarse geometric properties of more sophisticated mappings than taking powers. Namely, for any word $w\in F_d$ in the alphabet with $d$-symbols one may consider the associated \emph{word map} $\crse w\colon \crse G^d\to\crse G$. For instance, the word $w=[a_1,a_2]$ defines the mapping $\crse{G\times G\to G}$ sending $g_1,g_2\in\crse G$ to their (coarse) commutator $g_1\ast g_2\ast g_1^{-1}\ast g_2^{-1}$, while the word $w= a^n$ recovers the power map $\crse\mu _n$ we have been using all along. Properties such as coarse invertibility, triviality, surjectivity of a given word map $\crse w$ would then be invariant under coarse isomorphism.
\end{remark}

\begin{remark}
  The description of $\crse \mu_n$ given in \eqref{eq:mu as composition} makes it clear that the definition of $\minvariant[\mathbb N]\variable$ makes sense for group objects within any category, and all the properties we discussed above remain valid. For example, in the category of sets $\minvariant[\mathbb N]G$ becomes the set of exponents such that taking the $n$-th power defines a bijection of $G$ onto itself.
  In the special case where $G$ is a finite group (in the category of sets), the invariant $\minvariant[\mathbb N] G$ is just the set of numbers that are coprime to the order of the group.
\end{remark}

\section{Power invertibility spectra for certain word-metrics on $\mathbb Z$}
\label{sec:metric}
In this section we consider word-metrics $d_g$ on $\mathbb Z$ associated with the infinite generating sets $S_g \coloneqq \{\pm g^n : n \in \mathbb N\}$ for some integer $g>1$. An equivalent way of saying this is that we give $\mathbb Z$ the path-metric defined by the Cayley graph $\cay(\mathbb Z, S_g)$. This metric defines a coarse structure on $\mathbb Z$, which we denote by $\varcrs{\cay(S_g)}$. We may then rephrase our main question from the introduction as:

\begin{question}\label{qu:word-metric}
  Is it the case that $(\mathbb Z,\varcrs{\cay(S_{g_1})})$ and $(\mathbb Z,\varcrs{\cay(S_{g_2})})$ are isomorphic coarse groups if and only if $g_1^n=g_2^m$ for some positive integers $n,m\in\mathbb N$?
\end{question}

Main goal of the section will be the proof of \cref{thm:intro}, which relies on the computation of the power invertibility spectra of the coarse groups $(\mathbb Z, \varcrs{\cay(S_g)})$.
This provides strong evidence towards a positive answer to the \cref{qu:word-metric}, and positively answers the following.

\begin{question}[\cite{coarse_groups}*{Question 9.3.1}]\label{qu:word-metric_easy}
  Is it true that $(\mathbb Z,\varcrs{\cay(S_{2})})$ and $(\mathbb Z,\varcrs{\cay(S_{3})})$ are not coarsely isomorphic?
\end{question}

To compute $\minvariant{\mathbb Z,\varcrs{\cay(S_{g})}}$ it will be very convenient to use the $g$-adic representations constructed by Nathanson in \cite{nathanson2011problems}. Namely, we will use the following.

\begin{theorem}[\cite{nathanson2011problems}*{Theorems 3 and 6}]\label{thm:nat-representation}
  Given $g \geq 2$, every integer $k$ has a unique representation of the form
  \begin{equation*}
    k = \sum_{i = 0}^\infty \epsilon_i g^i
  \end{equation*}
  such that
  \begin{enumerate}
    \item $\epsilon_i \in \{0, \pm 1, \ldots, \pm\floor{g/2}\}$ for all $i$,
    \item $\epsilon_i \neq 0$ for only finitely many $i$,
    \item\label{item:even-rep} if $\abs{\epsilon_i} = {g/2}$, then $|\epsilon_{i+1}| \neq {g/2}$ and $\epsilon_i\epsilon_{i+1} \geq 0$
  \end{enumerate}
  (condition \ref{item:even-rep} is vacuous if $g$ is odd).
  Moreover, $k$ has word length
  \begin{equation*}
    \ell_g(k) = \sum_{i = 0}^\infty \abs{\epsilon_i}.
  \end{equation*}
\end{theorem}

The representations as in \cref{thm:nat-representation} are called \emph{special $g$-adic representations}. The equality $\ell_g(k) = \sum_{i = 0}^\infty \abs{\epsilon_i}$ can be understood as saying that the special $g$-adic representations define geodesic paths from $0$ to $k$ in $\cay(\mathbb Z,S_g)$.

\begin{remark}
  In fact, the `moreover' part is the main result of \cref{thm:nat-representation}. This is proven by showing that any writing of $k$ as a sum of $\pm g^i$ (\emph{i.e.}\ a path from $0$ to $k$ in $\cay(\mathbb Z,S_g)$) can be algorithmically reduced to a special $g$-adic representation without increasing the total number of addends (\emph{i.e.}\ the new path is not longer than the original one). Geodesicity then follows from uniqueness.
\end{remark}

The special $g$-adic representations are fairly stable within congruence classes modulo $g^n$. Explicitly, when $\epsilon_n(x)$ denotes the $n$-th coefficient of the special $g$-adic representation of an integer $x$, we have
\begin{equation}\label{eq: congruence mod gn}
  \epsilon_n(x) = \epsilon_n(y) \quad \text{for all} \quad x,y \in \mathbb Z \quad \text{with} \quad x \equiv y \pmod{g^{n+2}}.
\end{equation}
This is easily seen when $g$ is odd, as the special $g$-adic representation is just a different choice of representatives for the usual $g$-adic series expansion, and \eqref{eq: congruence mod gn} is then trivially satisfied (it is even true for $x,y \in \mathbb Z$ with $x \equiv y \pmod{g^{n+1}}$).
When $g$ is even the situation is a little more delicate, as $\pm g/2$ represent the same congruence class modulo $g$.

Let then $g$ be even. To prove \cref{eq: congruence mod gn} it is enough to show that $\epsilon_n(x) =\epsilon_n(y)$ when $y$ is the unique representative of $x$ modulo $g^{n+2}$ in $\{0,\ldots,g^{n+2}-1\}$.
Estimating the geometric series by absolute value, we observe
\begin{equation*}
  \Bigl| \sum_{i=0}^{n+1} \epsilon_i(x) g^i \Bigr| \leq \frac{g}{2} \sum_{i=0}^{n+1} g^i < g^{n+2}.
\end{equation*}By the above inequality the value $y$ is given by either
\begin{equation*}
  y = \sum_{i=0}^{n+1} \epsilon_i(x) g^i\quad \text{or} \quad y = \sum_{i=0}^{n+1} \epsilon_i(x) g^i + g^{n+2},
\end{equation*}
depending on whether $\sum_{i=0}^{n+1}\epsilon_i(x) g^i$ is positive or not.
In the first case, the coefficients $\epsilon_i(x)$ and $\epsilon_i(y)$ are equal for all $i\leq n+1$, as the truncated series is the unique special $g$-adic representation of $y$.

In the second case, the series $\sum_{i=0}^{n+1} \epsilon_i(x) g^i + g^{n+2}$ need not be a special $g$-adic representation of $y$. In fact, when $g=2$ the coefficients $\epsilon_{n+1}(x)$ and $1$ (from $g^{n+2}$) might violate condition \ref{item:even-rep}. If $|\epsilon_{n+1}(x)| \neq g/2$ the series is a special $g$-adic representation of $y$, and we deduce that $\epsilon_i(x)$ and $\epsilon_i(y)$ are equal for all $i\leq n+1$ as before.
When $|\epsilon_{n+1}(x)| = g/2$ we have $|\epsilon_{n}(x)| \neq g/2$. Since
\begin{equation*}
  y = \sum_{i=0}^{n} \epsilon_i(x) g^i + (\epsilon_{n+1}(x) + g) g^{n+1},
\end{equation*}
we can replace $\epsilon_{n+1}(x) + g$ by its special $g$-adic representation to obtain a special $g$-adic representation for $y$ (as $\epsilon_{n}(x) \neq g/2$ ensures that condition \ref{item:even-rep} holds for the newly added coefficients). This proves that $\epsilon_i(x) = \epsilon_i(y)$ for every $i\leq n$, and \cref{eq: congruence mod gn} follows.

\smallskip

This observation allows us to generalize the special $g$-adic representation to ring of $g$-adic integers $\mathbb Z_g = \varprojlim_{i=0}^\infty\mathbb Z/g^i\mathbb Z$, as the coefficients stabilize when taking projective limits modulo $g^n$. We record these observations in the following.

\begin{lemma}\label{prop:nathanson for Zq}
  Given $g \geq 2$, every $x\in \mathbb Z_g$ has a unique special $g$-adic representation
  \(
    x = \sum_{i = 0}^\infty \epsilon_i(x) g^i
  \)
  such that
  \begin{enumerate}
    \item $\epsilon_i \in \{0, \pm 1, \ldots, \pm\floor{g/2}\}$ for all $i$,
    \item if $\abs{\epsilon_i} = {g/2}$, then $|\epsilon_{i+1}| \neq {g/2}$ and $\epsilon_i\epsilon_{i+1} \geq 0$.
  \end{enumerate}
  Moreover, if $x \equiv y \pmod{g^n}$ for $x,y \in \mathbb Z_g$, then $\epsilon_i(x) = \epsilon_i(y)$ for all $i \leq n - 2$.
\end{lemma}

Of course, a $g$-adic number $x \in \mathbb Z_g$ is an integer if and only if $\epsilon_i(x) \neq 0$ for only finitely many indices $i$.
This allows us to make the following, fairly simple observation.

\begin{lemma}\label{lem:word length in Zq}
  Let $x_n$ be a sequence of integers, and suppose that $x_n$ converges to some $x \in \mathbb Z_g$.
  If $x\notin \mathbb Z$ then the lengths $\ell_g(x_n)$ diverge to infinity.
\end{lemma}

\begin{proof}
  Let $x = \sum_{i = 0}^\infty \epsilon_i(x) g^i$ be the special $g$-adic representation of $x$.
  Since the values of $x_n$ converge to $x$, for every $N \in \mathbb N$ there exists an $n_0$ such that $x_n \equiv x \pmod{g^N}$ for all $n \geq n_0$.
  That is, if $\epsilon_i(x_n)$ is the special $g$-adic representation of $x_n$, then $\epsilon_i(x_n) = \epsilon_i(x)$ for all $n\geq n_0$ and $i \leq N - 2$.
  
  If $x$ is not an integer then the sum $\sum_{i = 0}^\infty \abs{\epsilon_i(x)}$ diverges and hence the length $\ell_g(x_n)\geq \sum_{i = 0}^{N-2} |\epsilon_i(x_n)|$ also goes to infinity.
\end{proof}

Now that the preliminaries are in place, it will be easy to compute the power invertibility spectra of $(\mathbb Z,\varcrs{\cay(S_g)})$. We split the proof in the following two lemmas.

\begin{lemma}\label{lem:Z Sq not invertible}
  Let $g \geq 2$ be an integer and $p$ a prime with $(p,g) = 1$. Then the multiplication map $\mu_p \colon \mathbb (\mathbb Z, \varcrs{\cay(S_g)}) \to \mathbb (\mathbb Z, \varcrs{\cay(S_g)})$ is not coarsely invertible.
\end{lemma}
\begin{proof}
  To prove that the map $\mu_p$ is not coarsely invertible, it is sufficient to find a bounded subset $C\subset \mathbb Z$ such that $\mu_p^{-1}(C)$ is not bounded (cf.\ \cref{rmk: crse equivalence are proper}), where bounded is a synonym for ``finite $d_g$-diameter''.
  We choose the set
  \begin{equation*}
    C = \bigl\{ g^{(p-1) i} - 1 : i \in \mathbb N \bigr\},
  \end{equation*}
  which is obviously bounded because $\ell_g(g^n - 1) \leq 2$ for all $n \geq 0$.

  Since $(p,g) = 1$, we obtain
  \begin{equation*}
    g^{(p-1)i} - 1 \equiv 0 \pmod p,
  \end{equation*}
  which shows that $p$ divides $g^{(p-1)i} - 1$. We thus have:
  \begin{equation*}
    \mu_p^{-1}(C) = \biggl\{ \frac{g^{(p-1)i} - 1}{p}\ :\ i \in \mathbb N \biggr\}\subset \mathbb Z.
  \end{equation*}

  For every $n\geq 1$, $p$ is invertible modulo $g^n$. The congruence class of $p^{-1}$ modulo $g^n$ is therefore well defined, and we can view $p^{-1}$ as an element of $\mathbb Z_g$.  
  Of course, we then have
  \begin{equation*}
    \frac{g^{(p-1)i} - 1}{p} 
    \equiv p^{-1}g^{(p-1)i} - p^{-1} 
    \equiv -p^{-1} \pmod{g^{(p-1)i}}.
  \end{equation*}
  If we let $x_i \coloneqq (g^{(p-1)i} - 1)/p\in \mu_p^{-1}(C)$, the above congruence implies that the $x_i$ converge to $-p^{-1} \in \mathbb Z_g$ as $i \to \infty$. 
  Since $1/p$ is not an integer, \cref{lem:word length in Zq} implies that the word lengths $\ell_g(x_i)$ diverge to infinity, and therefore $\mu_p^{-1}(C)$ is not bounded in $(\mathbb Z, \varcrs{\cay(S_g)})$.
\end{proof}

\begin{lemma}\label{lem:Z Sq invertible}
  Let $g \geq 2$ be an integer. Then $k\mapsto \floor{k/g}$ is a coarse inverse for $\mu_g\colon (\mathbb Z, \varcrs{\cay(S_g)}) \to (\mathbb Z, \varcrs{\cay(S_g)})$.
\end{lemma}
\begin{proof}
  We show that $\floor{\variable/g}$ is a contraction. Let $k, k'\in\mathbb Z$ be arbitrary, and let $n=d_{g}(k,k')$. This means that there are sequences $k=k_0,\ldots ,k_n=k'$ and $m_i\in\mathbb N$ such that $\abs{k_i-k_{i-1}}= g^{m_i}$ for every $i=1,\ldots,n$.
  Observe that if $m_i\geq 1$ then $k_i$ and $k_{i-1}$ are in the same residue class mod $g$, hence 
  \[
    \abs*{\floor*{\frac{k_i}{g}}-\floor*{\frac{k_{i-1}}{g}}}
    =\abs*{\frac{k_i}{g}-\frac{k_{i-1}}{g}}
    = g^{m_{i-1}}.
  \]
  On the other hand, if $m_i=0$ then ${k_i}/{g}$ and $k_{i-1}/g$ either coincide or they differ by at most one. Using the triangle inequality along the path $\floor{k_i/g}$ shows that $d_{g}(\floor{k/g},\floor{k'/g})\leq n$, as desired.

  This proves that $\floor{\variable/g}$ is a controlled map. It is clear that it is a coarse inverse for $\mu_g$, because $\floor{\variable/g}\circ\mu_g$ is the identity and $\mu_g\circ \floor{\variable/g}$ is within distance $g$ from it.
\end{proof}

It is now immediate to compute the power invertibility spectrum of $(\mathbb Z, \varcrs{\cay(S_g)})$. Namely, we prove the following.

\begin{theorem}\label{thm:invariant for Z Sq}
  For every $g \in \mathbb N$ with $g \geq 2$ we have
  \begin{equation*}
    \minvariant{\mathbb Z, \varcrs{\cay(S_g)}} = \{ p \in \mathcal P : p \mid g \}.
  \end{equation*}
\end{theorem}

\begin{proof}  
    \cref{lem:Z Sq invertible} shows that $g \in \minvariant[\mathbb N]{\mathbb Z, \varcrs{\cay(S_g)}}$. Since $\minvariant[\mathbb N]\variable$ is closed under division, it follows that if $p\mid g$ then $p\in \minvariant{\mathbb Z, \varcrs{\cay(S_g)}}$.
    Conversely, if $p \nmid g$ then $p \notin \minvariant{\mathbb Z, \varcrs{\cay(S_g)}}$ by \cref{lem:Z Sq not invertible}.
\end{proof}

\begin{corollary}[\cref{thm:intro}]\label{cor: word-metric crse iso}
  $(\mathbb Z,\varcrs{\cay(S_{g_1})})$ and $(\mathbb Z,\varcrs{\cay(S_{g_2})})$ are isomorphic coarse groups then $g_1$ and $g_2$ must have the same prime factors.
\end{corollary}

\begin{corollary}[Answer to \cref{qu:word-metric_easy}]
  $(\mathbb Z,\varcrs{\cay(S_{2})})$ and $(\mathbb Z,\varcrs{\cay(S_{3})})$ are not isomorphic coarse groups.
\end{corollary}

\section{Power invertibility spectra of profinite coarsifications of $\mathbb Z$}\label{sec:profinite}
One way of defining coarse structures on groups is by means of group topologies, and a natural class of topologies on groups are the profinite ones. In this section we compute the invertibility spectra for all the profinite coarsifications of $\mathbb Z$.
Let $Q\subseteq \mathcal P$ be a non\=/empty set of primes, and consider set of moduli
\[
  Q^{*}\coloneqq\braces{m\in\mathbb{N}\ :\ (m,p)=1 \text{ $\forall p$ prime }p\notin Q},
\]
\emph{i.e.}\ the positive numbers that are products of powers of primes in $Q$.
The family of quotients $\mathbb{Z}/m\mathbb{Z}$ with $m\in Q^{*}$ forms an inverse system of finite groups, whose limit is the \emph{pro-$Q$ completion} of $\mathbb{Z}$
\[
 \mathbb{Z}_Q\coloneqq \varprojlim_{m\in Q^{*}} \mathbb{Z}/m\mathbb{Z}.
\]
The group $\mathbb{Z}_Q$ is given the limit topology $\tau_Q$ (each finite quotient is seen as a discrete group), which is a metrizable, compact, group topology on $\mathbb{Z}_Q$.

For a given prime $p$, the \emph{pro\=/$p$ completion} is $\mathbb{Z}_p\coloneqq \mathbb{Z}_{\{p\}}$, \emph{i.e.} the group of $p$\=/adic integers equipped with the $p$-adic topology.
By the Chinese Reminder Theorem, the pro\=/$Q$ completion is seen to be isomorphic (as a topological group) to a product of pro\=/$p$ completions:
\[
 \mathbb{Z}_Q\cong\prod_{p\in Q}\mathbb{Z}_p.
\]

Now, for any choice of $Q$, the natural homomorphism $\mathbb{Z}\to\mathbb{Z}_Q$ is an embedding with dense image (when identifying $\mathbb{Z}_Q$ with the product of pro\=/$p$ completions, the embedding $\mathbb{Z}\hookrightarrow\prod_{p\in Q}\mathbb{Z}_p$ is the diagonal embedding).
In particular, $(\mathbb Z,\tau_Q)$ is an abelian topological group. This lets us define a coarse structure $\varcrs{Q}$ on $\mathbb Z$ by declaring that a family of subsets $(A_i)_{i\in I}$ of $\mathbb Z$ is uniformly bounded (a.k.a.\ $\varcrs{Q}$-controlled) if and only if there exists a subset $K\subset \mathbb Z$ that is $\tau_Q$-compact and such that for every $i\in I$ we have $A_i\subseteq k_i+K$ for some $k_i\in \mathbb Z$. Note that $\mathbb Z$ itself is \emph{not} bounded, as it is not compact in the pro-$Q$ topology.

It is shown in \cite{coarse_groups}*{Section 9.2} that $(\mathbb Z,\varcrs{Q})$ is a non-trivial connected coarse group. Moreover, it is also shown that $\varcrs{Q}=\varcrs{Q'}$ if and only if $Q=Q'$, hence these are $2^{\aleph_0}$ distinct coarsifications of $\mathbb Z$.
As for \cref{qu:word-metric}, this means that the identity function $\id\colon \mathbb Z\to\mathbb Z$ is not a coarse isomorphism, but does not answer the following:

\begin{question}[\cite{coarse_groups}*{Question 9.3.2}]\label{qu:profinite}
  Is it the case that $(\mathbb Z,\varcrs{Q})$ and $(\mathbb Z,\varcrs{Q'})$ are isomorphic coarse groups if and only if $Q=Q'$?
\end{question}

We will use power invertibility spectra to answer the above affirmatively as stated in \cref{thm:intro2}. Once again, the proof relies on two lemmas.

\begin{lemma}\label{lem:multiplication not proper}
    If $p\notin Q$, then $\mu_p\colon \mathbb Z\to\mathbb Z$ is not proper with respect to the $Q$-adic topology.
\end{lemma}

\begin{proof}
    Interestingly, the proof of this lemma is quite similar to the proof of \cref{lem:Z Sq not invertible}. We again rely on the observation that $p$ is invertible in $\mathbb Z_Q$, because it has an inverse modulo $m$ for every $m\in Q^*$.
    Therefore, we may pick a sequence of integers $a_n$ that converges to $p^{-1}\in \mathbb Z_Q {\setminus\,} \mathbb Z$ and let
  \begin{equation*}
    K \coloneqq \{ p a_n : n \in \mathbb N \}\cup\{1\}\subset \mathbb Z.
  \end{equation*}
    Since $p a_n$ converges to $1$ in $\mathbb Z_Q$, this set is compact. It also has close resemblance to the bounded set used in the proof of \cref{lem:Z Sq not invertible} and can be exploited in much the same way.

    Namely, we observe that the preimage of $K$ is given by
  \begin{equation*}
    \mu_p^{-1}(K) = \{a_n : n \in \mathbb N\}.
  \end{equation*}
  Since the $a_n$ converge to $p^{-1}$ and $p^{-1} \notin \mu_p^{-1}(K)$, this shows that $\mu_p^{-1}(K)$ is not compact.
\end{proof}

\begin{lemma}\label{lem:division is continuos}
  If $p\in Q$, then the function $\floor{\variable/p}\colon \mathbb Z\to\mathbb Z$ sending $k\mapsto\floor{\frac{k}{p}}$ is continuous with respect to the $Q$-adic topology.
\end{lemma}

\begin{proof}
We split $\mathbb Z$ into residue classes
\begin{equation*}
  X_i = \{x \in \mathbb Z : x \equiv i \pmod p\},
\end{equation*}
for $i=0,\ldots, p-1$, and observe that each $X_i$ is a clopen in the pro-$Q$ topology.
It is then sufficient to check that the restriction of $\floor{\variable/p}$ to each of the $X_i$ is continuous (in the $Q$-adic topology).

Since $\mathbb Z_Q\cong\prod_{q\in Q}\mathbb Z_q$ is given by a product topology, it is sufficient to check that the mapping $X_i \to \mathbb Z_q$ sending $x$ to $\floor{\tfrac{x}{p}}$ is continuous for all $q \in Q$.
This is immediate for $X_0 =p\mathbb Z$: here $\floor{{\variable}/{p}}$ is simply equal to $({\variable}/{p})$, and the claim can be verified \emph{e.g.}\ using that the $q$-adic topology is induced by the $q$-adic absolute value $|\variable|_q$.
The claim follows easily for the other $X_i$ as well, because the restriction of $\floor{{\variable}/{p}}$ to $X_i$ equals the composition of $x\mapsto x-i$ and $({\variable}/{p})\colon X_0\to \mathbb Z_q$.
\end{proof}

We may now compute the power invertibility spectra:

\begin{theorem}\label{thm:invertible iff p in Q}
  For every $Q\subseteq\mathcal P$ we have $\minvariant{\mathbb Z, \varcrs Q} = Q$.
\end{theorem}

\begin{proof}
  If $p\notin Q$, we claim that $\mu_p\colon (\mathbb Z,\mathcal E_Q)\to (\mathbb Z, \mathcal E_Q)$ is not proper (as a coarse map). In fact, by \cref{lem:multiplication not proper} there exists a compact (and hence $\mathcal E_Q$-bounded) subset $K\subseteq(\mathbb Z, \tau_Q)$ whose preimage $\mu_p^{-1}(K)$ is not compact in $(\mathbb Z, \tau_Q)$. Since $\mu_p$ is continuous, $\mu_p^{-1}(K)$ is already closed, so it cannot be relatively compact (the proof of \cref{lem:multiplication not proper} also shows explicitly that it has an accumulation point at infinity). This means that $\mu_p^{-1}(K)$ is not $\mathcal E_Q$-bounded, proving our claim.

  Let now $p\in Q$. We claim that $\floor{\variable/p}\colon(\mathbb Z,\mathcal E_Q)\to (\mathbb Z, \mathcal E_Q)$ is a coarse inverse for $\mu_p$. Observe that $\floor{\variable/p}\circ \mu_p =\id_{\mathbb Z}$ and that $\mu_p\circ\floor{\variable/p}$ only differs from $\id_{\mathbb Z}$ by elements in $\{0,\dots,p-1\}$ which is finite (and hence bounded). 
  If we prove that $\floor{\variable/p}$ is controlled, this show that it is a coarse inverse for $\mu_p$.

  This can be verified directly, observing that for every compact $K\subset \mathbb Z$ and $k\in \mathbb Z$ we have
  \[
      \floor*{\frac{k+K}{p}}\subseteq \floor*{\frac{k}{p}} + \left(\floor*{\frac{K}{p}} \cup \Bigparen{\floor*{\frac{K}{p}}+1} \right).
  \]
  Since $\floor*{{K}/{p}}$ is compact by \cref{lem:division is continuos}, this shows that uniformly $\mathcal E_Q$-bounded sets are sent to uniformly $\mathcal E_Q$-bounded sets.
\end{proof}

\begin{corollary}[\cref{thm:intro2}]
  If $Q\neq Q'$ then $(\mathbb Z,\varcrs{Q})$ and $(\mathbb Z,\varcrs{Q'})$ are not coarsely isomorphic.
\end{corollary}

\begin{remark}
  The second part of the proof of \cref{thm:invertible iff p in Q} can be simplified a little using some more theory. Namely, since $\crse \mu_p$ is a coarsely surjective coarse homomorphism, it follows from \cite{coarse_groups}*{Proposition 5.2.11} that it is a coarse isomorphism if and only if it is proper (as a coarse map).
  \cref{lem:division is continuos} easily implies that $\mu_p\colon \mathbb Z\to \mathbb Z$ is topologically proper when $p\in Q$, and properness as a coarse map is then an immediate consequence.
\end{remark}

\section{Open Problems}\label{sec:problems}

As explained, a complete solution to \cref{qu:word-metric} remains elusive.
Since the case for $(\mathbb Z, \varcrs{\cay(S_2)})$ and $(\mathbb Z, \varcrs{\cay(S_3)})$ is now settled, the ``easiest'' case that remains open is:

\begin{problem}
    Is it true that $(\mathbb Z,\varcrs{\cay(S_{6})})$ and $(\mathbb Z,\varcrs{\cay(S_{12})})$ are not coarsely isomorphic?
\end{problem}

We list below other problems concerning coarsifications of $\mathbb Z$.
As in \cref{sec:profinite}, let $Q^*$ be the set of products of powers of elements in $Q\subseteq\mathcal P$.
As in \cref{sec:metric}, we may then consider the induced word metric and thus obtain a coarse group $(\mathbb Z, \varcrs{Cay(Q^*)})$.
It is shown in \cite{nathanson2011geometric}*{Theorem 3} (and also follows from \cite{jarden2007sums}, see \cite{hajdu2012representing}) that if $Q$ a non-empty finite set of primes then $(\mathbb Z, \varcrs{Cay(Q^*)})$ is not bounded, and is hence a non-trivial coarse group.

\begin{problem}
  Given $Q_1,Q_2\subset \mathcal P$ non-empty and finite, is it the case that $(\mathbb Z, \varcrs{Cay(Q_1^*)})$ and $(\mathbb Z, \varcrs{Cay(Q_2^*)})$ are coarsely isomorphic if and only if $Q_1=Q_2$?
\end{problem}

The relation between profinite and word-metric coarsifications of $\mathbb Z$ is also not well understood. Specifically, we ask the following.

\begin{problem}[cf. \cite{coarse_groups}*{Problem 9.3.3}]
    Given a non-empty set of primes $Q$, is $(\mathbb Z,\varcrs{Q})$ ever coarsely isomorphic to $(\mathbb Z,\varcrs{Cay(S)})$ for some (infinite) $S\subset\mathbb Z$? That is, is $\varcrs{Q}$ ever a \emph{geodesic coarse structure} \textup{(\cite{coarse_groups}*{Definition 2.2.13.})}?
\end{problem}

If $S$ is a finite generating set of $\mathbb Z$, it is a classical result that every coarse homomorphism $(\mathbb Z, \varcrs{Cay(S)})\to(\mathbb Z, \varcrs{Cay(S)})$ is within bounded distance from a linear map $\mathbb R\to \mathbb R$. We do not know whether more exotic coarsifications of $\mathbb Z$ allow for much wilder behaviour. For instance, we do not know the answer to the following version of \cite{coarse_groups}*{Problem 9.3.5}.

\begin{problem}
  Given $(\mathbb Z,\mathcal E)$ and $(\mathbb Z,\mathcal E')$ coarsifications of the kind discussed in this paper (word metric or profinite), can there exist a coarse homomorphism $(\mathbb Z,\mathcal E)\to(\mathbb Z,\mathcal E')$ whose image is neither bounded nor coarsely surjective?
\end{problem}

\appendix
\section{Appendix: coarse vs.\ bi-Lipschitz equivalence}
\label{appendix}

In number-theoretic contexts, it is far more common to encounter bi-Lipschitz equivalences as opposed to coarse equivalences. Fortunately, it is not hard to show that in the setting of \cref{sec:metric} these two notions coincide. The following arguments are fairly well-known and only included for the convenience of the interested reader.

The first observation is that if $(G,d_G)$ and $(H,d_H)$ are graphs equipped with their graph metrics, and $f\colon G\to H$ is a controlled map, then an iterated application of the triangle inequality shows that $f$ is Lipschitz with Lipschitz constant $\sup\braces{d_H(f(x),f(x')) : x,x'\in G,\ d_G(x,x')=1}$.
If a coarse equivalence $\crse {f\colon G\to H}$ admits a bijective representative $f\colon G\to H$, then both $f$ and $f^{-1}$ are controlled and hence Lipschitz, which proves that $(G,d_G)$ and $(H,d_H)$ are bi-Lipschitz equivalent (this argument applies equally well to all uniformly discrete coarsely geodesic metric spaces).

The question then becomes whether every coarse equivalence between graphs admits a bijective representative. The answer is negative in general, but it is true for graphs that locally have ``enough
space'' (see \cite{whyte_amenability_2008}). Since the Cayley graphs that we considered in this paper have infinite degree, it is not hard to see that every coarse equivalence has a bijective representative. Namely, the following is true:

\begin{lemma*}\label{lem:coarse is injective}
    Let $(G, d_G)$ and $(H, d_H)$ be a metric spaces with $G, H$ countable.
    Assume there exist disjoint covers $G = \bigsqcup_i B_i$ and $H = \bigsqcup_i C_i$ by infinite sets $B_i$ and $C_i$ of uniformly bounded diameter.
    Then every coarse equivalence $\crse{f\colon G \to H}$ has a bijective representative.
\end{lemma*}

\begin{proof}
    Pick any representative $f\colon G \to H$ of the coarse equivalence. Since the sets $C_i$ form a disjoint cover of $H$, we have $ G = \bigsqcup\nolimits_i f^{-1}(C_i)$. The sets $f^{-1}(C_i)$ are either finite or countable and the sets $C_i$ are infinite. Hence, can choose injective maps $g_i \colon f^{-1}(C_i) \to C_i$. These maps combine to an injective map $g\colon G \to H$. Since $f(x)$ and $g(x)$ belong to the same $C_i$, they are uniformly close (since the $C_i$ are uniformly bounded). Hence, $g$ and $f$ are close maps.

    The same construction yields an injective map $\tilde g \colon H \to G$ that is a representative for the coarse inverse $\crse f^{-1}$. By the Cantor–Schröder–Bernstein Theorem there exists a bijective map $h \colon G \to H$ with either $h(x) = g(x)$ or $h(x) = \tilde g^{-1}(x)$. This map is close to $g$ (and therefore a representative for $\crse f$) since
    \[
      d_H(g(x), h(x)) \leq \sup_{x \in \tilde g(G)} d_H(g(x), \tilde g^{-1}(x)) = \sup_{y \in H} d_H(g\circ \tilde g (y), y),
    \]
    which is bounded since $\tilde g$ is a coarse inverse of $g$.
\end{proof}

To apply the above lemma in the setting of \cref{sec:metric}, it is enough to exhibit an appropriate partition for $\cay\paren{\mathbb Z, S_g}$. This can be done abstractly, by choosing a maximal 3-separated set of vertices $X\subset \mathbb Z$ and partitioning $\mathbb Z$ as $\bigsqcup_{x\in X}B_x$, where each integer $n$ is contained in $B_x$ for the closest $x$ with ambiguities resolved arbitrarily.
Alternatively, an explicit partition can be constructed as
\[
  B_n = (n + S_g) \setminus \bigcup_{i=0}^{n-1} (i + S_g).
\]
These sets are disjoint by definition, and they are infinite because the intersection $(S_g + x) \cap (S_g + y)$ is finite for $x \neq y$ (which is seen by observing that $g^n \pm g^m$ takes any fixed non-zero value at most finitely many times).

\bibliography{BibCoarseGroups}
\end{document}

%% file: GenCommands.tex
\usepackage{amsmath, amsthm, amssymb}
\usepackage[T1]{fontenc}  
\usepackage[utf8]{inputenc}     
\usepackage{xcolor}
\usepackage{enumerate}
\usepackage[
  colorlinks=true, linkcolor=blue, linkbordercolor=blue, citecolor=red, citebordercolor=red, urlcolor=blue, linktocpage=true
]{hyperref}
\usepackage[capitalise, nameinlink, noabbrev, nosort]{cleveref} 
\usepackage[lite]{amsrefs}  

\usepackage{mathtools}          
\usepackage{microtype}          
\usepackage{bbm}          
\usepackage{bm}           
\usepackage[shortcuts]{extdash}       
\usepackage{subcaption}         
\usepackage{accents}
\usepackage{xspace}

\usepackage{xparse} 

\usepackage{tikz-cd}

\newcommand{\cay}{{\rm Cay}}

\DeclarePairedDelimiter\abs{\lvert}{\rvert}		

\DeclarePairedDelimiter\paren{(}{)}			
\DeclarePairedDelimiter\braces{\{}{\}}			
\DeclarePairedDelimiter\floor{\lfloor}{\rfloor}

	\newcommand{\Bigparen}[1]{\paren[\Big]{#1}}

\DeclareMathOperator{\id}{id}				

\theoremstyle{plain}
\newtheorem{thm}{Theorem}[section]		\newtheorem{theorem}[thm]{Theorem}
		
			\newtheorem{lemma}[thm]{Lemma}
		\newtheorem{corollary}[thm]{Corollary}
			\newtheorem{question}[thm]{Question}
		
		\newtheorem{problem}[thm]{Problem}

\newtheorem*{thm*}{Theorem}			\newtheorem*{theorem*}{Theorem}
\newtheorem*{prop*}{Proposition}		\newtheorem*{proposition*}{Proposition}
\newtheorem*{lem*}{Lemma}			\newtheorem*{lemma*}{Lemma}
\newtheorem*{cor*}{Corollary}			\newtheorem*{corollary*}{Corollary}
\newtheorem*{qu*}{Question}			\newtheorem*{question*}{Question}
\newtheorem*{conj*}{Conjecture}			\newtheorem*{conjecture*}{Question}
\newtheorem*{prob*}{Problem}		\newtheorem*{problem*}{Problem}
\newtheorem*{fact*}{Fact}
\newtheorem*{claim*}{Claim}
\newtheorem*{case*}{Case}

\numberwithin{equation}{section}

\newtheorem{alphthm}{Theorem}			

\theoremstyle{definition}

\newtheorem*{de*}{Definition}			\newtheorem{definition*}{Definition}
\newtheorem*{notation*}{Notation}
\newtheorem*{conv*}{Convention}			\newtheorem*{convention*}{Convention}

\theoremstyle{remark}
			\newtheorem{remark}[thm]{Remark}
			\newtheorem{example}[thm]{Example}

\DeclareMathAlphabet{\mathbit}{OT1}{cmr}{bx}{it}

\crefname{thm}{Theorem}{Theorems}              \crefname{theorem}{Theorem}{Theorems}
\crefname{prop}{Proposition}{Propositions}     \crefname{proposition}{Proposition}{Propositions}
\crefname{lem}{Lemma}{Lemmas}                  \crefname{lemma}{Lemma}{Lemmas}
\crefname{rmk}{Remark}{Remarks}                \crefname{remark}{Remark}{Remarks}
\crefname{cor}{Corollary}{Corollaries}         \crefname{corollary}{Corollary}{Corollaries}
\crefname{qu}{Question}{Questions}             \crefname{question}{Question}{Questions}
\crefname{conj}{Conjecture}{Conjectures}       \crefname{conjecture}{Conjecture}{Conjectures}
\crefname{prob}{Problem}{Problems}       \crefname{problem}{Problem}{Problems}
\crefname{fact}{Fact}{Facts}
\crefname{claim}{Claim}{Claims}
\crefname{case}{Case}{Cases}
\crefname{alphthm}{Theorem}{Theorems}          \crefname{alphcor}{Corollary}{Corollaries}
\crefname{alphprop}{Proposition}{Propositions}

\newcommand{\Cat}[1]{\ifmmode \text{\normalfont \textbf{#1}} \else {\normalfont \textbf{#1}}\fi}
\newcommand{\mhyphen}{\textnormal{-}}
\newcommand{\variable}{\,\mhyphen\,}

\definecolor{darkgreen}{rgb}{0.0, 0.5, 0.0}
\newcounter{fcomment}

\definecolor{brown}{rgb}{0.55, 0.3, 0.25}
\newcounter{xcomment}

\definecolor{newpurple}{rgb}{0.8, 0, 0.9}

%% file: CrsCommands.tex
\ExplSyntaxOn
\cs_new_protected:Nn \egreg_embolden_command:N
 {\cs_set_eq:cN { __egreg_ \cs_to_str:N #1 : } #1
  \cs_set_protected:Npn #1 { \bm { \use:c { __egreg_ \cs_to_str:N #1 : } } }}
\cs_new_protected:Nn \egreg_embolden_char:n
 {\exp_args:Nc \mathchardef { __egreg_#1: } = \mathcode`#1 \scan_stop:
  \cs_set_protected:cn { __egreg_#1_bold: } { \bm { \use:c { __egreg_#1: } } }
  \char_set_active_eq:nc { `#1 } { __egreg_#1_bold: }
  \mathcode`#1 = "8000 \scan_stop:}
\cs_new_protected:Nn \egreg_embolden:
 {\clist_map_function:nN
   {
    A,B,C,D,E,F,G,H,I,J,K,L,M,N,O,P,Q,R,S,T,U,V,W,X,Y,Z,
    a,b,c,d,e,f,g,h,i,j,k,l,m,n,o,p,q,r,s,t,u,v,w,x,y,z,
    *,<,>,/,-,1,2,3,4,5,6,7,8,9,0
   }
   \egreg_embolden_char:n
  \clist_map_function:nN
   {
    \alpha,\beta,\gamma,\delta,\epsilon,\varepsilon,\zeta,\eta,\theta,\vartheta,\iota,\kappa,\lambda,\mu,\nu,\xi,\pi,\varpi,\rho,\varrho,\sigma,\varsigma,\tau,\upsilon,\phi,\varphi,\chi,\psi,\omega,
    \Gamma,\varGamma,\Delta,\varDelta,\Theta,\varTheta,\Lambda,\varLambda,\Xi,\varXi,\Pi,\varPi,\Sigma,\varSigma,\Upsilon,\varUpsilon,\Phi,\varPhi,\Psi,\varPsi,\Omega,\varOmega,
    \cup,\cap,
    \neq,\cong,\ncong,
    \in,
    \subset,\subseteq,\supset,\supseteq,\subsetneq,\supsetneq,\nsubseteq,\nsupseteq,
    \leq,\geq,\lneq,\gneq,\lhd,\rhd,\trianglelefteq,\trianglerighteq,\triangleright,\trianglerighteq,\circ,
   }
   \egreg_embolden_command:N}
\NewDocumentCommand{\crse}{m}
 {\group_begin:
  \egreg_embolden:
  #1
  \group_end:}
\ExplSyntaxOff

\newcommand*{\oldcrse}[1]{\bm{#1}}

\newcommand{\cop}{\mathbin{\raisebox{-0.1 ex}{\scalebox{1.2}{$\boldsymbol{\ast}$}}}}

\newcommand*{\cinversefn}{[\mhyphen]^{\oldcrse{-1}}}

\makeatletter
\def\varcrs{\@ifnextchar[{\@withvarcrs}{\@withoutvarcrs}}
\def\@withvarcrs[#1]#2{\mathcal{E}_{\rm #2}^{\rm #1}}
\def\@withoutvarcrs#1{\mathcal{E}_{\rm #1}}
\def\varFcrs{\@ifnextchar[{\@withvarFcrs}{\@withoutvarFcrs}}
\def\@withvarFcrs[#1]#2{\mathcal{F}_{\rm #2}^{\rm #1}}
\def\@withoutvarFcrs#1{\mathcal{F}_{\rm #1}}
\makeatother